\documentclass[leqno,11pt]{amsart}
\usepackage{amssymb,amsmath,latexsym,amsthm}
\usepackage{verbatim}
\numberwithin{equation}{section}

\def\ga{\mathfrak{a}}

\def\gk{\mathfrak{k}}

\def\C{\mathbb{C}}

\def\F{\mathbb{F}}

\def\H{\mathbb{H}}

\def\N{\mathbb{N}}

\def\R{\mathbb{R}}

\def\Z{\mathbb{Z}}

\def\cH{\mathcal{H}}

\def\Re{{\rm Re}\,}

\def\rank{{\rm rank}\,}

\theoremstyle{plain}

\newtheorem{theorem}{Theorem}

\newtheorem{lemma}[theorem]{Lemma}

\newtheorem{remark}[theorem]{Remark}

\newcommand{\fa}{\mathfrak{a}}

\newcommand{\fg}{\mathfrak{g}}

\newcommand{\fk}{\mathfrak{k}}

\newcommand{\fn}{\mathfrak{n}}

\newcommand{\fs}{\mathfrak{s}}

\newcommand{\bc}{\mathbf{c}}
\newcommand{\rS}{\mathrm{S}}

\newcommand{\pr}{\mathrm{proj}}

\newcommand{\SO}{\mathrm{SO}}
\newcommand{\Sp}{\mathrm{Sp}}
\newcommand{\Spin}{\mathrm{Spin}}
\newcommand{\SU}{\mathrm{SU}}
\newcommand{\U}{\mathrm{U}}
\newcommand{\diag}{\mathrm{diag}}

\newcommand{\gL}{\Lambda }

\newcommand{\bmid}{\,\, \right|\,\,}

\def\mfg{\mathfrak{g}}

\begin{document}
\title{Direct Systems of Spherical Functions and Representations}

\author{Matthew Dawson}
\address{Department of Mathematics, Louisiana State University, Baton Rouge, LA 70803, U.S.A.}
\email{mdawso5@math.lsu.edu}
\thanks{The research of M. Dawson was partially supported by DMS-0801010 and DMS-1101337}
\author{Gestur \'{O}lafsson}
\address{Department of Mathematics, Louisiana State University, Baton Rouge,
LA 70803, U.S.A.}
\email{olafsson@math.lsu.edu}
\thanks{The research of G. \'Olafsson was supported by NSF grants DMS-0801010 and DMS-1101337, and the Max Planck Institute, Bonn, Germany}

\author{Joseph A. Wolf}
\address{Department of Mathematics, University of California, Berkeley,
CA 94720-3840, U.S.A.}
\email{jawolf@math.berkeley.edu}
\thanks{The research of J. A. Wolf was partially supported by a grant from the
Simons Foundation}

\subjclass[2000]{43A85, 53C35, 22E46}
\keywords{Injective limits; Compact symmetric spaces;
Spherical representations; Spherical functions}

\begin{abstract}
Spherical representations and functions are the building blocks
for harmonic analysis on riemannian symmetric spaces.  Here we consider
spherical functions and spherical
representations related to certain infinite dimensional symmetric
spaces $G_\infty/K_\infty = \varinjlim G_n/K_n$. We use the representation
theoretic construction $\varphi (x) = \langle e, \pi(x)e\rangle$ where
$e$ is a $K_\infty$--fixed unit vector for $\pi$.  Specifically, we look
at representations $\pi_\infty = \varinjlim \pi_n$ of $G_\infty$ where
$\pi_n$ is $K_n$--spherical, so the spherical representations
$\pi_n$ and the corresponding spherical functions $\varphi_n$ are related by
$\varphi_n(x) = \langle e_n, \pi_n(x)e_n\rangle$ where $e_n$ is a
$K_n$--fixed unit vector for $\pi_n$, and we consider the possibility of
constructing a $K_\infty$--spherical function $\varphi_\infty = \lim \varphi_n$.
We settle that matter by proving the equivalence of
(i)  $\{e_n\}$ converges to a nonzero $K_\infty$--fixed vector $e$, and
(ii) $G_\infty/K_\infty$ has finite symmetric space rank  (equivalently,
it is the
Grassmann manifold of $p$--planes in $\F^\infty$ where $p < \infty$ and
$\F$ is $\R$, $\C$ or $\H$).   In that finite rank case we also prove the
functional equation
\[\varphi(x)\varphi(y) =
\lim_{n\to \infty} \int_{K_n}\varphi(xky)dk\]
of
Faraut and Olshanskii, which is their definition of spherical functions.
We use this, and recent
results of M. R\"osler, T. Koornwinder and M. Voit, to show that in the case of finite rank all $K_\infty$-spherical representations of $G_\infty$ are given by the above limit formula. This in
particular shows that the characterization of the spherical representations in terms of highest weights
is still valid as in the finite dimensional case.
\end{abstract}

\maketitle

\section{Introduction}
\setcounter{equation}{0}
\setcounter{theorem}{0}
\noindent
Representation theory and harmonic analysis on symmetric spaces is by
now well understood. The building blocks are the spherical representations
and the corresponding spherical functions. For the case of a compact
symmetric space $G/K$ the spherical representations are parameterized
by a certain semi-lattice $\Lambda$.  When $G/K$ is simply connected,
$\Lambda$ is described by the Cartan-Helgason theorem. For each
$\mu\in\Lambda$ let $(\pi_\mu, V_\mu)$ denote the corresponding irreducible
representation. Then the space $V_\lambda^K$ of $K$-fixed vectors has
dimension $1$. Let $e_\mu$ be a unit vector in $V_\mu^K$. The function
\begin{equation}\label{eq-sphericalIntro1}
\psi_\mu (x)=\langle e_\mu ,\pi_\mu (x)e_\mu \rangle
\end{equation}
does not depend on the choice of $e_\mu$\,,  and
$\{\sqrt{\dim V_\mu}\psi_\mu\}_{\mu\in\Lambda }$ is an orthonormal basis
for $L^2(G/K)^K$.  In particular
\[f=\sum_{\mu\in\Lambda} \dim V_\mu \langle f,\psi_\mu \rangle \psi_\mu\]
for every $f\in L^2(G/K)^K$. Here the sum is taken in the $L^2$-sense.
Similarly, every $f\in L^2(G/K)$ can be written as
\[f=\sum_{\mu\in \Lambda }
   \dim V_\mu \langle\pi_\mu (f)e_\mu ,  \pi_\mu (\, \cdot \, ) e_\mu \rangle\]
where $\pi_\mu (f)=\int_G f(gK)\pi_\mu (g) \, d(gK)$.
\medskip

The function $\psi_\mu$ is spherical in the sense that
\begin{equation}\label{eq-sphericalIntro}
\int_K \psi_\mu (xky)\, dk=\psi_\mu (x)\psi_\mu (y)\quad
	\text{ for all } x,y\in G\,.
\end{equation}
Here $dk$ is normalized Haar measure on the compact group $K$.
Every positive definite spherical function on $G$ is obtained in this way
from an irreducible unitary representation of $G$.
\medskip

It is natural to extend the study to infinite dimensional Lie groups and
symmetric spaces. The simplest case is $G_\infty =\varinjlim G_n$,
$K_\infty=\varinjlim K_n$ and $M_\infty =\varinjlim G_n/K_n$ where
$G_n\subseteq G_{n+1}$ is a sequence of compact Lie groups such that
$K_n = G_n \cap K_{n+1}$.  The basic theory was developed by G. Olshanskii
(see \cite{Ol1990}) for the classical direct limits and for a very important
class of representations; by L. Natarajan,
E. Rodr\i guez--Carrington and one of us for more general direct limits
(see \cite{NRW1}, \cite{NRW2} and \cite{NRW3}); and by
S. Str\u atil\u a \& D. Voiculescu (see their survey \cite{SV}) for the
factor representation viewpoint.  See J. Faraut
\cite{F2008} for further information and references.
\medskip

The equation (\ref{eq-sphericalIntro}) does not make sense here because
there is no invariant measure on $K_\infty$. The replacement is the
functional equation
\begin{equation}\label{eq-sphericalIntro2}
\lim_{n\to \infty}\int_{K_n} \psi (xk_n y)\, dk_n=\psi (x)\psi (y)  \quad \text{ for all } x,y\in G_\infty\, .
\end{equation}
Again, see \cite{F2008}, the function $\psi$ is spherical if and only if
there is an irreducible unitary representation $(\pi, V)$ of $G_\infty$
and $e\in V^{K_\infty}$ with $\|e\|=1$, such that $\psi$ is given by
(\ref{eq-sphericalIntro1}).
\medskip

On the other hand, limits of irreducible spherical representations for a
strict direct system $\{M_n=G_n/K_n\}$ of compact symmetric spaces
were studied by the last named author in a series of articles
\cite{W2011}, \cite{W2009}, and \cite{W2010}, and then later by the last
two authors in \cite{OW2009} and \cite{OW2010}.  In particular, in
\cite{OW2009}, \cite{OW2010} and \cite{OW2011} they introduced the notion
of propagation of symmetric spaces. In short, if the $G_n$ are compact
and connected, and $\pi_{n}$ is a spherical representation of $G_n$,
then there exists in a canonical way a spherical representation $\pi_{{n+1}}$
of $G_{n+1}$ such that $\pi_{n}$ is a subrepresentation of
$\pi_{n+1}|_{G_n}$ with
multiplicity $1$. Furthermore, if $u_{n+1}$ is a highest weight vector for
$\pi_{{n+1}}$ then $\pi_{n}$ is realized as $\pi_{n+1}|_{G_n}$ acting on
the space generated by $\pi_{{n+1}}(G_n)u_{n+1}$. The system
$\{(\pi_{n}, V_n)\}$
is injective and $(\pi_\infty, V_\infty):= \varinjlim (\pi_n,V_n)$ is an
irreducible unitary representation of $G_\infty$.
\medskip

In this article we address the question of whether the representation
$(\pi_\infty ,V_\infty)$ is spherical. Our main result is Theorem
\ref{maintheorem} below.  It states that
 $V_\infty^{K_\infty}\not=\{0\}$ if and only if the symmetric space ranks of
the compact riemannian symmetric spaces $M_n$ are bounded.
Thus $V_\infty^{K_\infty}\not=\{0\}$ only for the symmetric spaces
$\SO(p+\infty)/\SO(p)\times \SO(\infty)$,
$\SU(p+\infty)/\mathrm{S}(\U(p)\times \U(\infty))$ and
$\Sp(p+\infty)/\Sp(p)\times \Sp(\infty)$ where $0 < p < \infty$.
We then show that if $\{e_n\}$ is a sequence
of $K_n$-invariant vectors in $V_n$ of norm $1$ and
$e=\lim_{n\to \infty} e_n\in V_\infty^{K_\infty}$, then the function
$\psi_\infty (x):=\langle e , \pi_\infty (x)e\rangle$ is spherical in
the sense of (\ref{eq-sphericalIntro2}), and
\[\psi_\infty (x)=\lim_{n\to \infty} \psi_n (x)\]
where $\psi_n (x)=\langle e_n, \pi_n (x)e_n\rangle$.  See Theorem
\ref{mainTheorem2}.

Further discussion of the finite rank case is given in Section \ref{sec7}.
Using result from \cite{RKV} we show that in the case of
finite rank all $K_\infty$-spherical representations of $G_\infty$ are
given by the limit construction in Section \ref{sec3}.  This in
particular shows that the characterization of the spherical representations
in terms of highest weights
is still valid in the finite rank case.

\section{Propagation of Symmetric Spaces}\label{sec1}
\setcounter{equation}{0}
\setcounter{theorem}{0}

\noindent
In this section we give a short overview of injective limits and propagation
of compact symmetric spaces, as needed for our considerations on
limits of spherical representations. We refer
to \cite{OW2010} and \cite{W2010} for details.
\medskip

Let $M=G/K$ be a riemannian symmetric space of  compact
type. Thus $G$ is a connected semisimple compact Lie group with an involution
$\theta$ such that
\[(G^\theta)_o\subseteqq K\subseteqq G^\theta\]
where $G^\theta =\{x\in G\mid \theta (x)=x\}$ and the subscript ${}_o$ denotes
the connected component containing the identity element. For simplicity
we assume that $M$ is simply connected.
\medskip

Denote the Lie algebra of $G$ by $\fg$. By abuse of notation we write
$\theta$ for the involution $d\theta: \fg\to \fg$.  As usual
$\fg=\fk\oplus \fs$
where $\fk=\{X\in\fg\mid \theta(X)=X\}$ is the Lie algebra of $K$ and
$\fs=\{X\in \fg\mid \theta (X)=-X\}$.
Fix a maximal abelian subspace $\fa\subset \fs$.
For $\alpha \in\fa^*_{_\C}$ let
\[
\fg_{_\C,\alpha} =\{X\in\fg_{_\C} \mid [H,X]=\alpha (H)X \text{ for all }
H\in \fa_{_\C}\}\, .
\]
If $\fg_{_\C,\alpha}\not=\{0\}$ then $\alpha$ is called a (restricted) root.
Denote by $\Sigma=\Sigma (\fg,\fa)\subset i\fa^*$ the set of roots.  Let
$\Sigma_0=\Sigma_0(\fg,\fa)=\{\alpha \in\Sigma\mid 2\alpha\not\in\Sigma\}$,
the set of nonmultipliable roots.
Then $\Sigma_0$ is a root system in the usual sense and the Weyl
group corresponding to $\Sigma (\fg, \fa)$ is the same as the Weyl group
generated by the reflections $s_\alpha$, $\alpha \in \Sigma_{0}$.
Furthermore, $M$ is irreducible as a riemannian symmetric space if and only
if $\Sigma_{0}$ is irreducible as a root system.
\medskip

Let $\Sigma^+\subset \Sigma$ be a positive system and
$\Sigma^+_{0}=\Sigma^+ \cap \Sigma_{0}$. Then
$\Sigma^+_{0}$ is a positive system in
$\Sigma_{0} $. Denote
by $\Psi =\{\alpha_1,\ldots ,\alpha_r\}$, $r=\dim \fa$, the set of simple roots in $\Sigma_{0}^+$. Since we will be dealing with direct limits we may
assume that
$\Sigma$, and hence $\Sigma_0$, is one of the classical root systems.
In order to facilitate considerations of direct limits,
we number the simple roots in the following way:

\begin{equation}\label{rootorder}
\begin{aligned}
&\begin{tabular}{|c|l|c|}\hline
$\Psi=A_r$&
\setlength{\unitlength}{.5 mm}
\begin{picture}(155,18)
\put(5,2){\circle{4}}
\put(2,5){$\alpha_{r}$}
\put(6,2){\line(1,0){13}}
\put(24,2){\circle*{1}}
\put(27,2){\circle*{1}}
\put(30,2){\circle*{1}}
\put(34,2){\line(1,0){13}}
\put(48,2){\circle{4}}
\put(49,2){\line(1,0){23}}
\put(73,2){\circle{4}}
\put(74,2){\line(1,0){23}}
\put(98,2){\circle{4}}
\put(99,2){\line(1,0){13}}
\put(117,2){\circle*{1}}
\put(120,2){\circle*{1}}
\put(123,2){\circle*{1}}
\put(129,2){\line(1,0){13}}
\put(143,2){\circle{4}}
\put(140,5){$\alpha_1$}
\end{picture}
&$r\geqq 1$
\\
\hline
$\Psi =B_r$&
\setlength{\unitlength}{.5 mm}
\begin{picture}(155,18)
\put(5,2){\circle{4}}
\put(2,5){$\alpha_{r}$}
\put(6,2){\line(1,0){13}}
\put(24,2){\circle*{1}}
\put(27,2){\circle*{1}}
\put(30,2){\circle*{1}}
\put(34,2){\line(1,0){13}}
\put(48,2){\circle{4}}
\put(49,2){\line(1,0){23}}
\put(73,2){\circle{4}}
\put(74,2){\line(1,0){13}}
\put(93,2){\circle*{1}}
\put(96,2){\circle*{1}}
\put(99,2){\circle*{1}}
\put(104,2){\line(1,0){13}}
\put(118,2){\circle{4}}
\put(115,5){$\alpha_2$}
\put(119,2.5){\line(1,0){23}}
\put(119,1.5){\line(1,0){23}}
\put(143,2){\circle*{4}}
\put(140,5){$\alpha_1$}
\end{picture}
&$r\geqq 2$\\
\hline
$\Psi=C_r$ &
\setlength{\unitlength}{.5 mm}
\begin{picture}(155,18)
\put(5,2){\circle*{4}}
\put(2,5){$\alpha_{r}$}
\put(6,2){\line(1,0){13}}
\put(24,2){\circle*{1}}
\put(27,2){\circle*{1}}
\put(30,2){\circle*{1}}
\put(34,2){\line(1,0){13}}
\put(48,2){\circle*{4}}
\put(49,2){\line(1,0){23}}
\put(73,2){\circle*{4}}
\put(74,2){\line(1,0){13}}
\put(93,2){\circle*{1}}
\put(96,2){\circle*{1}}
\put(99,2){\circle*{1}}
\put(104,2){\line(1,0){13}}
\put(118,2){\circle*{4}}
\put(115,5){$\alpha_2$}
\put(119,2.5){\line(1,0){23}}
\put(119,1.5){\line(1,0){23}}
\put(143,2){\circle{4}}
\put(140,5){$\alpha_1$}
\end{picture}
& $r\geqq 3$
\\
\hline
$\Psi=D_r$ &
\setlength{\unitlength}{.5 mm}
\begin{picture}(155,20)
\put(5,9){\circle{4}}
\put(2,12){$\alpha_{r}$}
\put(6,9){\line(1,0){13}}
\put(24,9){\circle*{1}}
\put(27,9){\circle*{1}}
\put(30,9){\circle*{1}}
\put(34,9){\line(1,0){13}}
\put(48,9){\circle{4}}
\put(49,9){\line(1,0){23}}
\put(73,9){\circle{4}}
\put(74,9){\line(1,0){13}}
\put(93,9){\circle*{1}}
\put(96,9){\circle*{1}}
\put(99,9){\circle*{1}}
\put(104,9){\line(1,0){13}}
\put(118,9){\circle{4}}
\put(113,12){$\alpha_3$}
\put(119,8.5){\line(2,-1){13}}
\put(133,2){\circle{4}}
\put(136,0){$\alpha_1$}
\put(119,9.5){\line(2,1){13}}
\put(133,16){\circle{4}}
\put(136,14){$\alpha_2$}
\end{picture}
& $r\geqq 4$
\\
\hline
\end{tabular}
\end{aligned}
\end{equation}

The classical irreducible symmetric spaces are given by the following
table.  For the grassmannians we always assume that $p\leqq q$ and we
let $n=p+q$.
For $\alpha\in \Sigma$ we write $m_\alpha =\dim \fg_{_\C, \alpha}$, and
for the simple roots we write
$m_j=m_{\alpha_j}$.  For the realization of each root system see \cite{Ar1962},
\cite[Chapter 10]{He1978} or \cite{OW2009}.  In all these classical cases
$m_{\alpha_j/2} = 0$ for $j > 1$.  We will go into more detail
in Section \ref{sec2}.

{\footnotesize
\begin{equation}\label{symmetric-case-class}
\begin{tabular}{|c|l|l|c|c|c|c|} \hline
\multicolumn{7}{| c |}
{Irreducible compact riemannian symmetric $M = G/K$, $G$ classical, $K$ connected}\\
\hline \hline
\multicolumn{1}{|c}{} &  \multicolumn{1}{|c}{$G$} &
        \multicolumn{1}{|c}{$K$} &
        \multicolumn{1}{|c|}{$\Psi$}& $\begin{matrix}
                                        m_j \\ j>1
                                       \end{matrix}$& $m_1$& $m_{\alpha_1/2}$  \\ \hline \hline
$1$ &$\SU (n)\times \SU(n)$ & $\diag\, \SU(n)$ & $A_{n-1}$ &$2$&$2$&$0$ \\ \hline
$2$ &   $\begin{matrix}\Spin (2n+1)\times\\ \Spin (2n+1)\end{matrix}$ & $\diag\, \Spin (2n+1)$ &
    $B_n$&$2$&$2$&$0$  \\ \hline
$3$ & $\begin{matrix}\Spin (2n)\times\\ \Spin (2n)\end{matrix}$ & $\diag\, \Spin(2n)$ &
    $D_n$&$2$&$2$&$0$   \\ \hline
$4$ &$\Sp (n)\times \Sp (n)$ & $\diag\,\Sp (n)$ & $C_n$&$2$&$2$&$0$   \\ \hline
$5$ & $\SU(n)$ & $\mathrm{S}(\U (p)\times \U ( q))$ &
    $C_p$ & $2$&$1$& $2(q-p)$\\ \hline
$6$ & $\SU (n)$ & $\SO (n)$ & $A_{n-1}$
&$1$&$1$&$0$ \\ \hline
$7$ & $\SU (2n)$ & $\Sp (n)$ & $A_{n-1}$&$4$&$4$&$0$ \\ \hline
$8$ &$\SO (n)$ & $\SO (p) \times \SO (q)$ &
    $B_p$&$1$&$q-p$&$0$  \\ \hline
$9_1$  &$\SO (4n)$ & $\U (2n)$ & $C_{n}$ &$4$& $1$&$ 0$\\ \hline
$9_2$  &$\SO (2(2n+1))$ & $\U (2n+1)$ & $C_{n}$ &$4$& $1$&
$4$\\  \hline
$10$  &$\Sp (n)$ & $\Sp (p) \times \Sp (q)$ & $C_p$& $4$ & $3$ & $4(q-p)$\\ \hline
$11$   &$\Sp (n)$ & $\U (n)$ & $C_n$ & $1$& $0$& $0$\\ \hline
\end{tabular}
\end{equation}
}

Cases (5), (8), and (10) are the grassmannians of $p$--planes in $\F^n$,
$n=p+q$, where $\F=\C$, $\R$ or $\H$, respectively.  In cases (5) and (10),
$m_{\alpha_1/2}= (q-p)d$ and $m_{\alpha_1} = d-1$ where $d = \dim_\R \F$. It is therefore
more natural to view (8) as of type $C_p$ with $m_{\alpha_1/2}=q-p$ and $m_{\alpha_1}=d-1=0$.
\medskip

We now assume that $\{M_k=G_k/K_k\}_{k\geqq 1}$ is a sequence of compact symmetric spaces such
that $G_n\subseteqq G_k$ and $K_n = G_n\cap K_k$ for $n\leqq k$. Then
$M_n\subseteqq M_k$. We write $\Sigma_n$,
$\Sigma_n^+$, $\Sigma_{0,n}$, $\Psi_n$, etc.
when we need to indicate dependence on the symmetric
space $M_n$. We say that
$M_k$ \textit{propagates} $M_n$ if (i) $\fa_k=\fa_n$, or (ii) by choosing
$\fa_n\subseteqq \fa_k$ we obtain the Dynkin diagram in Table \ref{rootorder}
for $\Psi_k$ from that of $\Psi_n$ by only adding simple roots at the
left end.  Then in particular $\Psi_{n}$ and
$\Psi_{k}$ are of the same type.
\medskip

In \cite{OW2009}, \cite{OW2010}
and \cite{OW2011} we used the set of indivisible roots instead of the set
of nonmultipliable roots. Both definitions are equivalent.
\medskip

When $\fg_k$ propagates $\fg_n$, and $\theta_k$ and
$\theta_n$ are the corresponding involutions with
$\theta_k|_{\fg_n} = \theta_n$, the corresponding eigenspace decompositions
$\fg_k=\fk_k\oplus \fs_k$ and $\fg_n=\fk_n\oplus \fs_n$
give us
\[
\fk_n=\fk_k\cap \fg_n\, ,\quad \text{and}\quad \fs_n=\fg_n\cap \fs_k\, \quad
 \text{for}\quad k \geqq n.\]
We recursively choose maximal commutative subspaces $\fa_k\subset \fs_k$ such
that $\fa_{n} \subseteqq \fa_k$ for $k\geqq n$. We then have
$\Sigma_n\subseteqq \Sigma_{k}|_{\fa_n}\setminus \{0\}$.
We choose the positive ordering such that $\Sigma_n^+\subseteqq \Sigma_k^+|_{\fa_n}\setminus
\{0\}$.
\medskip

Note that by moving along each row in Table \ref{symmetric-case-class} we
have a propagation of symmetric spaces.  In all cases except (5), (8) and (10)
the multiplicities remain constant, in fact less or equal to $4$.
\medskip

We set
\[G_\infty =\varinjlim G_n\, \quad K_\infty =\varinjlim K_n\, ,\text{ and } M_\infty=\varinjlim M_n=G_\infty /K_\infty\, .\]
In this paper we consider the question of whether
the inductive limit of $K_n$--spherical representations of $G_n$ is
$K_\infty$--spherical.
For that we need to recall the construction of inductive limits of spherical representations, the theory of highest weights of spherical representations and the Harish--Chandra $c$--function of the noncompact dual of $G_n$.

\section{Spherical Representations of Compact Groups}\label{sec2}
\setcounter{equation}{0}
\setcounter{theorem}{0}

\noindent
In this section we give a short overview of spherical representations, their highest weights, and connections
to propagation of symmetric spaces. Most of the material can be found in
\cite{OW2009}, \cite{OW2010}, \cite{W2010}, \cite{W2009} and \cite{W2011}.
The notation will be as in Section \ref{sec1},
and $G$ or $G_n$ will always stand for a connected compact group. If $k\geqq n$ then we assume that $G_n\subseteqq G_k$ and that
$M_k$ propagates $M_n$. We also assume that each of the symmetric spaces $M_n$ is simply connected.
We denote by $r_k$ and $r_n$ the respective real ranks
of $M_k$ and $M_n$. As always we fix compatible $K_k$-- and
$K_n$--invariant inner products on $\fs_k$, respectively $\fs_n$.
\medskip

For a representation $(\pi, V)$ of $G$ let
$V^K=\{u\in V\mid (\forall k\in K)\,\, \pi (k)u=u\}$. If $(\pi ,V)$ is
irreducible, then we say that $(\pi ,V)$, or simply $\pi$, is $K$--spherical,
or just spherical, if $V^K\not= \{0\}$. It is well known that $\pi$ is
spherical if and only if $\dim V^K=1$. Furthermore, in that case the
highest weight of $\pi$ is contained in $i\fa^*$. Let
\begin{equation}
\Lambda^+=\Lambda^+(G,K)=\left\{ \mu \in i\fa^*\,\left|\,
\tfrac{\langle \mu ,\alpha\rangle}{\langle \alpha, \alpha\rangle}
\in \Z^+ \quad \text{ for all } \quad \alpha \in \Sigma^+\right.\right\}\, .
\end{equation}

\begin{theorem}[Cartan--Helgason]\label{t-CH}  Let $(\pi ,V)$ be
an irreducible representation of $G$, and $\mu$ its highest weight.
Then the following are equivalent.
\begin{enumerate}
\item $(\pi,V)$ is spherical.
\item $\mu\in i\fa^*$ and $\mu \in \Lambda^+$.
\item The multiplicity of $(\pi, V)$ in $L^2(M)$ is $1$.
\item $ \displaystyle{\pi}$ is a subrepresentation of the
representation of  $G$  in  $L^2(M)$.
\end{enumerate}
\end{theorem}
\noindent See \cite[Theorem 4.1, p. 535]{He1984} for the proof.
\medskip

From now on, if $\mu \in \gL^+$ then $\pi_\mu$ denotes the irreducible spherical representation with highest weight $\mu$.
Define linear functionals $\xi_{j}\in i\ga^*$ by
\begin{equation}\label{fundclass1}
\frac{ \langle \xi_{i},\alpha_{j} \rangle }
{\langle \alpha_{j},\alpha_{j} \rangle} = \delta_{i,j} \quad \text{ for }\quad
1 \leqq j \leqq r\ \ .
\end{equation}
Then $\xi_1,\ldots ,\xi_r\in\gL^+$ and
\[\gL^+=\Z^+\xi_1+\ldots  + \Z^+\xi_r=\left\{\left. \sum_{j=1}^r n_j\xi_j\bmid n_j\in \Z^+\right\}\, .\]
The weights $\xi_j$ are called the
\textit{class 1 fundamental weights for}
$(\mfg,\gk)$.
Set $\Xi=\{\xi_{1},\ldots ,\xi_{r}\}$.  We always have

\[
\rho =\sum_{j=1}^r \rho_j \xi_j \quad \text{ with }\quad
\rho_j =\frac{1}{2}\left( m_{\alpha_j} +\tfrac{m_{\alpha_j/2}}{2}\right)\, .
\]
We write $a=\frac{1}{2}\left( m_{\alpha_1}+\frac{m_{\alpha_1/2}}{2}\right)$.
Using $m_{\alpha_j}=m_{\alpha_i}$ for $i,j\geqq 2$ we set
$b=\frac{1}{2}m_{\alpha_j}$, $j\geqq 2$. Then
\begin{equation}\label{eq-rho}
\rho=a\xi_1+b\sum_{j=2}^r\xi_j\, .
\end{equation}

We will need a particular formulation for each classical root system.
We identify $\fa$ with $\R^r$ so that, as usual,
$\fa=\{(x_{r+1},\ldots ,x_1)\mid x_1+\ldots + x_{r+1} =0\}$ if
$\Psi =A_r$ and otherwise $\fa=\R^r$. Set
$f_1=(0,\ldots ,0,1)$, $f_2=(0,\ldots, 0,1,0)$, \ldots ,
$f_n=(1,0,\ldots , 0)$ where $n=r+1$ for $A_r$ and otherwise $n=r$.
\medskip

For $\Psi=A_{r}$ we have $\Sigma_0^+ =\{f_j-f_i\mid 1\leqq i < j \leqq n\}$, and $\alpha_j =f_{j+1}-f_j$, $j=1,\ldots ,r$. We have
\[\xi_j=2\sum_{i=j+1}^{r+1} f_i\, .\]
Thus
\[\gL^+\simeq \{(m_r,m_{r-1},\ldots ,m_1,0)\in (2 \Z^+)^{r+1} \mid  m_i\leqq m_j \text{ if } i<j\}\, .\]
The multiplicities are constant,
equal to $m=1$, $2$ or $4$. Hence $a=b=1/2,1$, or $2$ and we have
\begin{equation}\label{eq-Arho}
\rho = a\sum_{j=1}^{r} \xi_j=m(r,r-1,\ldots , 1, 0)=2a \sum_{j=1}^{r+1}(j-1)f_j \, .
\end{equation}

If $\Psi$ is of type $B_r$ then we have $\Sigma^+_0 =
  \{f_j\mid j=1,\ldots ,r\}\cup \{f_j\pm f_i\mid 1\leqq i <j\leqq r\}$ and
$\Psi =\{\alpha_1=f_1\}\cup \{\alpha_{i}=f_{i}-f_{i-1}\mid i=2,
  \ldots ,r\}$. Thus
\[\xi_1=\sum_{j=1}^r f_j \text{ and } \xi_j=2\sum_{i=j}^r f_i\, , \,\,  j>1 \, .\]
In particular
\[\gL^+\simeq \{(m_r,\ldots ,m_1)\in (\Z^+)^r \mid m_i\leqq m_j
  \text{ and } m_j-m_i \text{ even for } i<j\}\, .\]
Finally we have
\begin{equation}\label{eq-Brho}
\rho=\sum_{j=1}^n\xi_j=(2r-1,2r-3,\ldots , 3,1)=\sum_{j=1}^r(2j-1)f_j
\end{equation}
in case (2). Case (8), which is the other possibility for $\Psi$ of type $B$,
will be considered in the discussion of $C_r$\,, as explained just after
Table \ref{symmetric-case-class}.
\medskip

If $\Psi$ is of type $C_r$ then we have $\Sigma_0^+=\{2f_j\mid j=1,\ldots ,r\}\cup
	\{f_j\pm f_i\mid 1\leqq i <j\leqq r\}$ and
$\Psi =\{\alpha_1=2f_1\}\cup \{\alpha_j=f_j-f_{j-1}\mid j=2,\ldots ,r\}$. Thus
\[\xi_j=2\sum_{i=j}^r f_i\]
and
\begin{equation}\label{eq-Crho}
\rho = 2a\sum_{j=1}^r f_j\, +\,2b\sum_{\nu =2}^r\sum_{j =\nu }^r f_j
=2\sum_{j=1}^r(a+b(j-1))f_j\, .
\end{equation}

There is just one case where $\Psi$ is of type $D_r$.  There
$a=b=1$. In that case we have $\alpha_1=f_1+f_2$ and
$\alpha_j=f_{j}-f_{j-1}$ for $j\geqq 2$. Thus
\[\xi_1=\sum_{i=1}^r f_i \, , \,\, \xi_2=-f_1 + \sum_{j=2}^r f_j \, ,
	\text{ and } \xi_j=2\sum_{i=j}^r f_i\, \text{ for } j\geqq 3\, .\]
That gives us
\begin{equation}\label{eq-Drho}
\rho=2\sum_{j=2}^r (j-1)f_j\, .
\end{equation}

Fix a $\mu\in \Lambda^+$ and let $(\pi_\mu,V_\mu)$ be the corresponding
spherical representation of $G$. Fix a highest weight
vector $u_\mu \in V_\mu$ and a $K$--fixed vector $e_\mu$. We assume that
$\|u_\mu\|=\|e_\mu\|=1$. For the following it is important to evaluate the
inner product $\langle u_\mu ,e_\mu\rangle$ in a systematic way so that we
can control it as we consider inductive limits of spherical representations
in the next section. The following is well known, but we include the proof
for completeness.
\medskip

First of all we always have
 $\langle u_\mu ,e_\mu\rangle \not=0$. We choose $u_\mu$ and $e_\mu$ so
that $\langle u_\mu ,e_\mu\rangle >0$.
\medskip

Let $\fg^\prime =\fk\oplus i\fs$. As $G$ is compact it is a linear group, thus
contained in a complex linear group $G_{_\C}$ with Lie algebra $\fg_{_\C}$.
Let $G^\prime$
 be the analytic subgroup of $G_{_\C}$ with Lie algebra $\fg^\prime$.
Note that the holomorphic extension of $\theta$ to $\fg_{_\C}$ restricted
to $\fg^\prime$ defines a Cartan involution on $\fg^\prime$. We also denote
this involution and the
 corresponding Cartan involution on $G^\prime$ by $\theta$.
Let $\overline{N}=\theta (N)$.  Then $G^\prime$ has a Iwasawa
decomposition (recall that we are assuming $G$ simply connected,
in particular $K$ is connected)
 \begin{equation}\label{eq-Iwasawa}
 G^\prime =KA^\prime N
 \end{equation}
 where $A^\prime =\exp (i \fa )$ and the Lie algebra of $N$ is
$\fn=\bigoplus_{\alpha \in\Sigma^+}\fg^\prime_\alpha$.
\medskip

 For $x\in G^\prime$ write $x=k(x)a(x)n(x)$ according to the Iwasawa decomposition (\ref{eq-Iwasawa}). We normalize the Haar measure on
 $\overline{N}$ such that
 \[\int_{\overline{N}} a(\bar n )^{-2\rho }\, d\bar n =1\, ,\]
 Then the integral
 \[\bc (\lambda )=\int_{\overline{N}} a(\bar n)^{-\lambda -\rho}\, d\bar n\]
 converges for all $\lambda \in \fa^*_{_\C}$ such that $\Re \langle \lambda , \alpha \rangle >0$ for all $\alpha \in \Sigma^+$. The function
 $\bc (\lambda )$ is the Harish--Chandra $c$--function. It has a meromorphic continuation to all of $\fa^*_{_\C}$ and is given by
 \[\bc (\lambda )=\frac{'c(\lambda )}{'c (\rho )}\]
 where $'c(\lambda )$ is   explicitly given by the Gindikin--Karpelevich product formula. In terms of $\Sigma_0^+$, we have
\begin{equation}
\label{eq:c}
'c(\lambda )= \prod_{\alpha \in \Sigma_0^+}\,  'c_\alpha (\lambda_\alpha)
\end{equation}
where
\begin{equation}
\label{eq:cbeta}
'c_\alpha (\lambda_\alpha )= \frac{2^{-2\lambda_\alpha } \; \Gamma(2\lambda _\alpha)}
{\Gamma\left(\lambda_\alpha+\frac{m_{\alpha /2}}{4}+\frac{1}{2}\right)
\Gamma\left(\lambda_\alpha +\frac{m_{\alpha /2}}{4}+\frac{m_{\alpha }}{2}\right)}\, ,\quad
\lambda_\alpha =\frac{\langle \lambda ,\alpha \rangle}{\langle \alpha  , \alpha \rangle}
\end{equation}
where $\Gamma$ is the Euler $\Gamma$--function
$\Gamma (x) := \int_0^\infty e^{-t} t^{x-1} dt$.
Formula (\ref{eq:cbeta}) looks slightly different from the usual formula
for the $c$--function as found for instance in \cite{He1984}, Ch. IV,
Theorem 6.4, where it is written in terms of positive indivisible roots
($\alpha\in\Sigma^+$ with $\alpha/2 \notin \Sigma^+$) rather than in terms of
positive nonmultipliable roots. The formula (\ref{eq:c}) was used in
\cite{OP2011}. The equivalence of the two formulas follows from the doubling
formula $\sqrt{\pi}\Gamma (2x)=2^{2x-1}\Gamma (x)\Gamma (x+\frac{1}{2})$
for the Gamma function.
\medskip

Let $M=Z_{K}(\fa )$. The following can be found in any standard
reference on symmetric spaces.
\begin{lemma}\label{le-intKN} Let $f\in L^1(K/M)$. Then
\[\int_{K/M} f(kM)\, d(kM) =\int_{\overline{N}} f(k(\bar n )M)\, a(\bar n)^{-2\rho}\, d\bar n\, .\]
\end{lemma}

Clearly the vector $\int_ K \pi_\mu (k)u_\mu \, dk$ is $K$--invariant,
and by taking the inner product with $e_\mu$ one gets
$\int_K \pi_\mu (k)u_\mu \, dk = \langle u_\mu , e_\mu\rangle e_\mu$.

\begin{lemma}\label{lemma1}
$\displaystyle \left \langle \int_{K} \pi_{\mu}(k)u_\mu dk , u_\mu
\right \rangle  = \int_K \langle  \pi_{\mu}(k)u_\mu , u_\mu\rangle \, dk
=\bc (\mu +\rho )$.
\end{lemma}

\begin{proof} The proof is a simple calculation using Lemma \ref{le-intKN}.
To simplify the notation we write $u=u_\mu$ and $\pi $ for $\pi_\mu$. The representation $\pi$ extends to a holomorphic representation of $G_{_\C}$. Hence $\pi (x)$ is well defined for $x\in G^\prime$. We note that $\pi (mn)u=u$ for all $n\in N$ and $m\in M$ (see \cite{He1984}, Theorem 4.1, p. 535). In particular
$\langle \pi (\bar n)u,u\rangle = \langle u, \pi (\theta (n)^{-1})u \rangle = \langle u, u\rangle =1$.
$$
\begin{aligned}
\int_{K/M} \langle \pi (k)u, u\rangle d(kM)
  &= \int_{\overline{N}}\langle  \pi (k(\overline{n}))u , u\rangle
	a(\overline{n})^{-2\rho }\, d\overline{n} \\
  &= \int_{\overline{N}} \langle \pi(\overline{n}a(\overline{n})^{-1}
	n')u, u\rangle\, a(\overline{n})^{-2\rho }\, d\overline{n} \quad \text{ where }
	n' \in N \\
  &= \int_{\overline{N}}\langle  \pi (\overline{n})
	u, u\rangle \, a(\overline{n})^{-\mu-2\rho }\, d\overline{n} \\
&=\int_{\overline{N}}a (\bar n)^{-\mu -2\rho }\, d\bar n \\
&=\bc (\mu +\rho )\, .
\end{aligned}
$$
That proves the Lemma.
\end{proof}

\begin{theorem}\label{lemma3}
Let $u_\mu$ and $e_\mu$ be as above.  Then
$\langle u_\mu , e_\mu \rangle =
\sqrt{\bc (\mu + \rho)}$.  In particular
\[ \int_{K } \pi_{\mu }(k)u_\mu \, dk =
\sqrt{\bc (\mu +\rho)} \, e_\mu\,.\]
\end{theorem}

\begin{proof} We use the same notation as in Lemma \ref{lemma1}.  By that lemma, we see that
$$
\begin{aligned}
   \bc(\mu+\rho) &= \left \langle \int_{K} \pi_{\mu}(k)u_\mu dk , u_\mu \right \rangle  \\
                                                    &= \langle \langle u_\mu, e_\mu \rangle e_\mu, u_\mu \rangle \\
                                                    &= \langle u_\mu, e_\mu \rangle ^2
\end{aligned}
$$
Thus we see that $\langle u_\mu, e_\mu \rangle = \sqrt{\bc(\mu + \rho)}$.  The rest of the result follows when we recall that $\int_K \pi_\mu (k)u_\mu \, dk = \langle u_\mu , e_\mu\rangle e_\mu$.
\end{proof}

\begin{theorem}\label{th-cFctProd} For $\alpha\in\Sigma_0^+$ let $x_\alpha :=\frac{1}{4}\left(m_{\alpha/2}+2\right)$ and $y_\alpha=\frac{1}{4}\left(m_{\alpha/2}+2m_\alpha\right)$. Let $\mu\in \Lambda^+$. Then $\mu_\alpha \in \Z^+$ for all $\alpha \in \Sigma^+_0$ and
\begin{equation}\label{eq-cFctProd}
\bc (\mu +\rho) =\prod_{\alpha\in \Sigma_0^+}\left(\left(1+\frac{x_\alpha}{\rho_\alpha}\right)\left(1+\frac{y_\alpha}{\rho_\alpha}\right)\right)^{-\mu_\alpha}
\cdot\prod_{j=0}^{\mu_\alpha-1}
\frac{\left(1+\frac{j}{2\rho_\alpha}\right)\left(1+\frac{\mu_\alpha+j}{2 \rho_\alpha}\right)}
{\left(1+\frac{j}{x_\alpha+\rho_\alpha}\right)\left(1+\frac{j}{y_\alpha+\rho_\alpha}\right)}
\end{equation}
where the product $\displaystyle \prod_{j=0}^{\mu_\alpha-1}$ is understood
to be $1$ if $\mu_\alpha=0$.
\end{theorem}
\begin{proof} By (\ref{eq:c}) and
(\ref{eq:cbeta}) we can write $\bc (\mu +\rho) =\prod_{\alpha\in\Sigma_0^+}\bc_\alpha (\mu_\alpha+\rho_\alpha) $ with
\[\bc_\alpha (\mu_\alpha +\rho_\alpha) =\frac{2^{-2\mu_\alpha} \Gamma(2(\mu_\alpha + \rho_\alpha))}{\Gamma (2\rho_\alpha)}
\frac{\Gamma (\rho_\alpha + x_\alpha)}{\Gamma(\mu_\alpha +\rho_\alpha+x_\alpha)}\frac{ \Gamma(\rho_\alpha+y_\alpha)}{\Gamma(\mu_\alpha+\rho_\alpha+y_\alpha)}\, .\]
Now, using that $\mu_\alpha\in \Z^+$ and $\Gamma (x+1)=x\Gamma (x)$, we get for $\mu_\alpha\not= 0$:
\begin{eqnarray*}
\Gamma(2(\mu_\alpha + \rho_\alpha))&=&\left(\prod_{j=1}^{2\mu_\alpha}(2(\mu_\alpha +\rho_\alpha)-j)\right)\Gamma (2\rho_\alpha)\\
&=& \left(\prod_{j=0}^{2\mu_\alpha-1}(2\rho_\alpha+j)\right)\Gamma (2\rho_\alpha)\\
&=&2^{2\mu_\alpha}\rho_\alpha^{2\mu_\alpha}\left(\prod_{j=0}^{\mu_\alpha-1}
\left(1+\frac{j}{2\rho_\alpha}\right)
\left(1+\frac{\mu_\alpha+j}{2\rho_\alpha}\right)\right)\, \Gamma (2\rho_\alpha)\, .
\end{eqnarray*}
Similarly
\[
\frac{\Gamma (\rho_\alpha + x_\alpha)}{\Gamma(\mu_\alpha +\rho_\alpha+x_\alpha)}\frac{ \Gamma(\rho_\alpha+y_\alpha)}{\Gamma(\mu_\alpha+\rho_\alpha+y_\alpha)}=\prod_{j=0}^{\mu_\alpha -1}\frac{1}{(\rho_\alpha+x_\alpha+j)(\rho_\alpha +y_\alpha +j)}\]
and the claim follows.
\end{proof}

\section{Inductive Limits of Spherical Representations}\label{sec3}
\setcounter{equation}{0}
\setcounter{theorem}{0}

\noindent
In this section we recall the construction of inductive limits of spherical
representations  \cite[Section 3]{W2010} and \cite{OW2010}. We always assume
that $k\geqq n$ and that we have a sequence $\{M_k=G_k/K_k\}$ such that
$M_k$ propagates $M_n$. The index $k$ (respectively $n$) will
indicate objects related to $G_k$ (respectively $G_n$).
As in \cite{W2011} or \cite{W2010}, our description of the root system and
the fundamental weights gives

\begin{lemma} \label{resmult} Assume that $M_k$ propagates $M_n$. Let
\smallskip

\centerline{$\Psi_n=\{\alpha_{n,1},\ldots \alpha_{n,r_n}\}$ and
$\Xi_n=\{\xi_{n,1},\ldots , \xi_{n,r_n}\}$}
\smallskip
\noindent
and similarly for $M_k$. Assume that $j\leqq r_n$.  Then
\begin{enumerate}
\item[{\rm 1.}]  $\alpha_{k,j}$ is the unique element of
$\Psi_{k}$ whose restriction to $\fa_n$ is $\alpha_{n,j}$.
\item[{\rm 2.}]\label{label3} If $\mu_n=\sum_{j=1}^{r_n}k_j\xi_{n,j}\in \Lambda_n^+$, then $\mu_k:=\sum_{j=1}^{r_n}k_j\xi_{k,j}\in\Lambda_k^+$
and $\mu_k|_{\fa_n}=\mu_n$.
\end{enumerate}
\end{lemma}

For $I=(k_1,\ldots ,k_{r_n})\in (\Z^+)^{r_n}$ define
$\mu_I:=\mu(I)=k_1\xi_{n,1}+\ldots +k_{r_n}\xi_{n,r_n}$.
Lemma \ref{resmult} allows us to form direct system of representations, as
follows.  For $\ell \in \N$ denote by $0_\ell = (0,\ldots ,0)$ the zero vector
in $\R^\ell$.  For
$I_{n}=(k_1,\ldots ,k_{r_n})\in (\Z^+)^{r_n}$ let
\begin{equation}\label{I-notation}
\begin{aligned}
\bullet\ &\mu_{I,n}=\mu (I_n)= \sum_{j=1}^{r_n}k_j\xi_{n,j}\in \Lambda^+_n;\\
\bullet\ &\pi_{I,n} = \pi_{\mu (I_n)} \text{ the corresponding spherical
representation};\\
\bullet\ &V_{I,n} = V_{\mu (I_n)}  \text{ a fixed Hilbert space for the
representation } \pi_{I,n};\\
\bullet\ &u_{I,n} = u_{\mu (I_n)}  \text{ a highest weight unit vector in }
V_{ I,n}\, , \,\, \| u_{I,n}\|=1;\\
\bullet\ &e_{I,n} = e_{\mu (I_n)} \text{ a } K_n\text{--fixed unit vector in }
V_{I,n}\text{ such that } \langle u_{I,n}, e_{I,n}\rangle >0
\end{aligned}
\end{equation}
Later, $I_n$ will be fixed and we will write $\pi_n$, $V_n$, $\mu_n$,
$u_n$, and $e_n$ without further comments.

\begin{theorem}\label{l-inductiveSystemOfRep}
Assume that $M_k$ propagates $M_n$.
Let $(\pi_{I,n},V_{I,n})$ be an irreducible spherical representation of $G_n$
with highest weight $\mu_{I,n}\in\Lambda^+_n$.
Let $I_k=(I_n,0_{r_k-r_n})$. Then the following hold.
\begin{enumerate}
\item[{\rm 1.}] The $G_n$--submodule of $V_{I,k}$ generated
by $u_{I,k}$ is irreducible and isomorphic to $V_{I,n}$.
\item[{\rm 2.}] The multiplicity of $\pi_{I,n}$ in
$\pi_{I,k}|_{G_n}$ is $1$.
\end{enumerate}
\end{theorem}

\begin{remark} {\rm From this point on, when $n\leqq k$ we will always
assume that the Hilbert space $V_{I,n}$ is realized inside $V_{I,k}$ as
the span of $\pi_{I,k}(G_n)u_{I,k}$, in other words by identification of
highest weight unit vectors. Thus we can then assume that $u_{I,n}=u_{I,k}$.
On the other hand we almost never have $e_{I,n}=e_{I,k}$ under this inclusion.
But we can always assume that $e_{I,k}=q(k,n)\, e_{I,n} + e_{k,n}^\perp$
where $q(k,n)=\langle e_{I,k},e_{I,n}\rangle >0$ and $\langle e_{I,n},e_{k,n}^\perp\rangle =0$. One of our
aims is to evaluate $\langle e_{I,k},e_{I,n}\rangle $ in terms of $c$--functions.}
\hfill $\diamondsuit$
\end{remark}

Theorem \ref{l-inductiveSystemOfRep}
allows us to define a inductive limit of spherical representation starting by a given spherical representation $(\pi_{I,n},V_{I,n})$ of $G_n$.
We have isometric embeddings $V_n = V_{I,n} \hookrightarrow V_{n+1}=V_{I,n+1}$ defined by the map
$u_n=u_{I,n} \mapsto u_{n+1}=u_{I,n+1}$. As $u_{I,n}$ is independent of $n$ we simply write $u$ for the fixed highest weight vector. Then the algebraic inductive limit $\varinjlim V_\infty$ is a pre--Hilbert space with inner product
$\langle v,w\rangle =\langle v,w\rangle_{V_k}$ if $v,w\in V_k$. This inner product is well defined as the embeddings $V_n\hookrightarrow V_k$ are
isometric. We denote by $\mu_\infty =\varinjlim \mu_{I,n}\in i\fa_\infty^*$
and  $V_{\infty ,\mu_\infty}=V_\infty$ the Hilbert space completion of
$\varinjlim V_n$.  Notice that $u\in V_\infty$.
\medskip

Our main theorem in this article is the following theorem:
\begin{theorem}\label{maintheorem} Let the notation be as above and assume that $\mu\not= 0$.
Then $V_\infty^{K_\infty}\not=\{0\}$ if and only if the ranks of
the compact riemannian symmetric spaces $M_n$ are bounded.
Thus $V_\infty^{K_\infty}\not=\{0\}$ only for the symmetric spaces
$\SO(p+\infty)/\SO(p)\times \SO(\infty)$,
$\SU(p+\infty)/\mathrm{S}(\U(p)\times \U(\infty))$ and
$\Sp(p+\infty)/\Sp(p)\times \Sp(\infty)$ where $0 < p < \infty$.
\end{theorem}

Here is our strategy.
First, if $V_\infty^{K_\infty}\not= \{0\}$ let
$e_\infty \in V_\infty^{K_\infty}$ be a unit vector. Then consider the projection
$\pr_{\infty, n}(e_\infty)\in V_n^{K_n}\setminus \{0\}$. Let
$e_n=\pr_{\infty, n} (e_\infty) /\|\pr_{\infty, n}(e_\infty)\|$.
Then $\{e_n\}$ is a Cauchy sequence such that $e_n$ is a unit vector in
$V_n^{K_n}$ and $\{e_n\} \to e_\infty$. On the other hand if $\{e_n\}$ is a
Cauchy sequence in $V_\infty$ such that $e_n\in V_n^{K_n}$ and $\|e_n\|=1$,
then $e_\infty = \lim e_n$ is a nonzero element of $V_\infty^{K_\infty}$.
Thus $V_\infty^{K_\infty}\not=\{0\}$ if and only if we can find a Cauchy
sequence $\{e_n\}$ such that $e_n\in V_n^{K_n}$ is a unit vector.
\medskip

Recursively choose $K_{n+1}$--fixed unit vectors $e_{n+1}$ so that the
orthogonal projection of $V_{n+1}$ onto $V_n$ sends $e_{n+1}$ to a positive
real multiple of $e_n$ as mentioned before.  Then
$\text{proj}_{m,n}(e_m) = q (m,n ) e_n$ for $m \geqq n$ where the
$q (m,n)$ are real with $0 < q (m,n) \leqq 1$.  Since
$\text{proj}_{m,\ell} = \text{proj}_{n,\ell}\cdot \text{proj}_{m,n}$
we have $q (n, \ell) q (m,n) = q (m,\ell )$. Furthermore, for a fixed $n$ the function $m\mapsto q(m,n)$ is decreasing.
Also, choosing $e_1$ such that
$\langle u, e_1 \rangle$ is positive real we have $\langle u, e_n \rangle$
positive real for all $n$.

\begin{theorem}\label{theorem4}
Let $m \geqq n$.  Then
$\langle e_m, e_n \rangle = \sqrt{\bc_m(\mu_m+\rho_m)/\bc_n(\mu_n+\rho_n)}$.
\end{theorem}

\begin{proof} To simplify the notation we write $\mu$ for both $\mu_m$ and
$\mu_n$ and $\bc_n (\mu+\rho)$ for $\bc_n (\mu_n+\rho_n)$. From Lemma
\ref{lemma3} we have $e_m = \bc_m(\mu+\rho)^{-1/2}\int_{K_m} \pi_{m}(k)u\, dk$
and similarly for $e_n$.  So
$$
\begin{aligned}
\langle e_m, e_n \rangle &= (\bc_m(\mu +\rho)\bc_n(\mu+\rho))^{-1/2}
	\int_{k \in K_n} \int_{h \in K_m} \langle \pi_{m}(h)u,
		\pi_{n}(k)u \rangle \, dk \, dh \\
&= (\bc_m(\mu+\rho )\bc_n(\mu +\rho))^{-1/2} \int_{k \in K_n} \int_{h \in K_m}
	\langle \pi_{m}(k^{-1}h)u, u \rangle \, dh \, dk \\
&= (\bc_m(\mu +\rho)\bc_n(\mu +\rho))^{-1/2} \int_{h \in K_m}
	\langle \pi_{m}(h)u, u \rangle \, dh \quad \text{ as } K_n\subseteqq K_m\\
&= (\bc_m(\mu +\rho)\bc_n(\mu +\rho))^{-1/2} \bc_m(\mu +\rho)\\
&	= \sqrt{\bc_m(\mu +\rho)/\bc_n(\mu +\rho)}
\end{aligned}
$$
as asserted.
\end{proof}

\begin{theorem}\label{theorem5} The limit
$\lim_{m\to \infty} \bc (\mu_m+\rho_m)$ exists and is non--negative.
Let the sequence $\{e_n\}_n$ be as before. Then $\{e_n\}$ converges to a
nonzero element
$e \in V_\infty^{K_\infty}$ if and only if $\lim \bc_m(\mu_m+\rho_m)>0$.
\end{theorem}

\begin{proof} We start by observing
$$
\| e_m - e_n\|^2 = \|e_n\|^2 -2\langle e_m, e_n \rangle + \|e_m\|^2=2(1-\langle e_m,e_n\rangle )\, .$$
Hence $\{e_n\}$ is a Cauchy sequence if and only if
\[\lim_{m,n\to \infty }\langle e_m,e_n\rangle
=\lim_{m,n\to \infty }\sqrt{\bc_m(\mu_m+\rho_m)/\bc_n(\mu_n+\rho_n)}=1\, .
\]
For fixed $n$ the sequence $\langle e_m ,e_n\rangle\geqq 0$ is decreasing
and bounded below by zero. Hence  $\ell_n:=\lim_m \langle e_m , e_n\rangle$
exists. This implies that the limit $\lim_m \bc (\mu_m+\rho_m)$ exists and is
non--negative.
\medskip

The sequence $0\leqq \ell_n \leqq 1$ is either zero or increasing
and hence $\lim \ell_n$ exists. It follows that
$ \displaystyle{ \lim_{m,n\to\infty}\sqrt{\bc_m(\mu_m+\rho_m)/\bc_n(\mu_n+\rho_n)}}$
exists (and thus has to be equal to $1$) if and only if  $\lim_{m\to \infty} \bc_m (\mu_m+\rho_m)>0$.
\end{proof}

\section{The Finite Rank Cases}\label{sec4}
\setcounter{equation}{0}
\setcounter{theorem}{0}

\noindent
In this section we prove Theorem \ref{maintheorem} for the finite rank
cases, i.e. the cases where $M_n =\SO (n)/\SO (p)\times \SO (q)$,
$ \SU(n)/\rS(\U (p)\times \U (q))$  or  $\Sp (n)/\Sp (p)\times \Sp (q)$
with $p$ fixed and $n = p + q$.  We may assume $q \geqq p$, so all the
$M_n$ have the same finite rank $p$.  The cardinality of
$\Sigma_0^+$ is constant.
\medskip

We use the notation from the previous section. Recall $d=\dim_\R \F$.
View the real grassmannian $\SO (n)/\SO (p)\times \SO (q)$ as of type
$C_p$ with $m_{\alpha_1}=0$ as explained after Table \ref{symmetric-case-class}.
\medskip

Note that $q\to \infty$ with $n$. Furthermore,
the highest weight $\mu=\sum_{\nu = 1}^p k_\nu\xi_\nu$ is independent of $n$.
Now, in view of Theorem \ref{th-cFctProd}, it suffices to prove that
\[\frac{x_\alpha}{\rho_\alpha}=\frac{1}{4}\left(\frac{2}{\rho_\alpha} +\frac{ m_{\alpha/2}}{\rho_\alpha}\right)\quad \text{ and }\quad
\frac{y_\alpha }{\rho_\alpha}=\frac{1}{4}\left(2\frac{m_\alpha}{\rho_\alpha} +\frac{m_{\alpha/2}}{\rho_\alpha}\right)\]
are  bounded for all $\alpha$.
\medskip

For that we only need to consider where $\alpha$ is in the Weyl group orbit
of $\alpha_1$, because in all other cases $m_\alpha$ and $m_{\alpha/2}$ are
bounded (in fact $\leqq 4$). We have
$$
m_{\alpha_1} = d-1\,,\quad m_{\alpha_1/2}= d(q-p)\, \quad \text{and} \quad
\rho_{\alpha_1} = \tfrac{1}{2}\left(d-1 +\tfrac{(q-p)d}{2}\right).
$$
Thus $\tfrac{m_{\alpha_1/2}}{\rho_{\alpha_1}}$ is bounded. That completes
the proof of Theorem \ref{maintheorem} for the finite rank cases.

\section{The Infinite Rank Cases.}\label{sec5}
\setcounter{equation}{0}
\setcounter{theorem}{0}

\noindent
In this section we prove Theorem \ref{maintheorem} for the infinite rank
cases, i.e., the cases where $\rank M_n$ is unbounded.  We may pass to a
subsequence of $\{M_n\}$, and then of $\{n\}$, and assume that
$\rank M_n = n$.  Now we start the proof
by reducing it to the case where $\mu_\alpha = 1$ in (\ref{eq-cFctProd}).

\begin{lemma}\label{le-FirstReduction} Assume that $\mu_m=\sum_{j=1}^n k_j\xi_{m,j}$ with $k_n>0$. Then
\[\bc_m(\mu_m +\rho_m)\le \bc_m(\xi_{m,n}+\rho_m)\, .\]
\end{lemma}

\begin{proof} By \cite[Corollary 6.6, Ch IV]{He1984} we have
$\mu_{m,n}(\log (a(\bar n ))) \leqq 0$, so $\bc_m(\mu_m +\rho_m)$
is a decreasing function of $\mu_m$.
\end{proof}

We will also need the following well known and simple fact:

\begin{lemma}\label{le-InfProd} Assume that $\epsilon, \delta >0$.
Let $a_j\geqq \epsilon $ and $ 0\leqq x_j\le \delta$. Then
\[\lim_{N\to \infty }\prod_{j=L}^N \left(1+\tfrac{a_j}{x_j+j}\right)^{-1}=0\, .\]
\end{lemma}
\begin{proof} If $x>0$ is small enough then $1+x\leqq e^{x/2}$. Hence
\[ \prod_{j=L}^N \left(1+\tfrac{a_j}{x_j+j}\right)\geqq
\exp \left(\epsilon {\sum}_{j=L}^N \tfrac{1}{\delta +j}\right)
\to \infty \text{ as } N\to \infty\]
and the claim follows.
\end{proof}

The idea of the proof of Theorem \ref{maintheorem}, for the unbounded
rank cases, is to find a sequence of roots $\alpha_n$ such
that $\xi_{n,k,\alpha_n}=1$ and $\rho_{n,\alpha_n}$ is affine linear
in $n$. Then, if $\alpha_n /2$ is not a root, $x_{\alpha_n} = 1/2$
and the expression for $\bc_{\alpha_n} (\xi_{n,k,\alpha_n}+\rho_{n,\alpha_n})$
in (\ref{eq-cFctProd}) reduces to
\[\bc_{\alpha_n}(\xi_{n,k,\alpha_n}+\rho_{n,\alpha_n})=\left(1+\frac{y_{\alpha_n}}{\rho_{n,\alpha_n}}\right)^{-1}\, .\]
It will then follow from Lemma \ref{le-InfProd} that
\[\lim_{N\to \infty}\prod_{n=k}^N\bc_{\alpha_n}(\xi_{n,k,\alpha_n}+\rho_{n,\alpha_n})=0\, .\]
That will finish the proof because
$\bc_{n,\alpha}(\xi_{n,k,\alpha}+\rho_{n,\alpha})\leqq 1$
for all $n$ and all positive roots.
\medskip

In the case $\Psi = A_n$ we let $\alpha_n = f_{n+1}-f_1$.
Then $\rho_{n,\alpha_n}=tn$, $t=1/2, 1$ or $2$,
$\xi_{n,k,\alpha_n}=1$, $x_{\alpha_n}=1/2$, and
$y_{\alpha_n}=t$, and the claim follows by the argument indicated above.
\medskip

If $\Psi = B_n$ we take $\alpha_n=f_n-f_1$ when $k\not= 1$ and
$\alpha_n=f_n$ when $k=1$. Then $x_{\alpha_n}=1/2$, $\xi_{n,k,\alpha_n}=1$,
$\rho_{n,\alpha_n}$ is affine linear in $n$, and the claim follows as in
the $A_n$ case using the argument indicated above.
\medskip

When $\Psi = C_n$ we take $\alpha_n=f_n+f_1$.  Then both the
multiplicities and the $\rho_{n,\alpha_n}$ increase affinely in $n$, to
the claim follows as above.
\medskip

In the one $D_n$ case we take
$\alpha_n=f_n+f_2$ for $n$ large and the same argument goes through.
This completes the proof of Theorem \ref{maintheorem}.
\hfill $\square$

\section{The Connection to Spherical Functions on $G_\infty$}\label{sec6}
\setcounter{equation}{0}
\setcounter{theorem}{0}

\noindent
Finally, we discuss the connection with
the theory of $K_\infty$--spherical functions on $G_\infty$ as developed by
Olshanskii, Faraut and coworkers. See \cite{F2008} for references. Those
authors define a nonzero continuous function $\varphi : G_\infty \to \C$
to be \textit{spherical} if for all $x,y\in G_\infty$
\begin{equation}\label{def-Spherical}
\lim_{n\to \infty}\int_{K_n} \varphi (xky)\, dk=\varphi(x)\varphi (y)\, .
\end{equation}
By taking $x$, respectively $y$ to be the identity it is clear that a
spherical function is $K_\infty$--biinvariant and takes the value $1$ at
the identity.
\begin{theorem}\label{mainTheorem2} Assume that rank $G_\infty /K_\infty$
is finite and $V_{\infty}=\varinjlim V_{n,\mu_n}$. Let $\{e_n\}_{n}$ be a
Cauchy sequence in $V_{\infty}$ such that for all $n$
$\|e_n\|=1$, $e_n\in V_{n,\mu_n}^{K_n}$ and
$e_n\to e_\infty \in V_{\infty }^{K_\infty}$. Then
\[\varphi_{\mu_\infty} (x):=\langle e_\infty ,
    \pi_{\infty ,\mu_\infty}(x)e_n\rangle=\lim_{n\to \infty }
    \langle e_n, \pi_{n,\mu_n}(x)e_n\rangle \]
is a positive definite $K_\infty$--spherical function on $G_\infty$ in
the sense of {\rm (\ref{def-Spherical})}.
\end{theorem}

\begin{proof} Write $e_\infty = e_n+e_n^\perp$. Let $x,y \in G_{j_o}$.
Then, for $j\geqq j_o$,
\[\varphi_{\mu_\infty} (x)=\langle e_j, \pi_j (x)e_j\rangle = \varphi_{\mu_j} (x)+\langle e_j^\perp , \pi_{j}(x)e_j^\perp\rangle\]
because $V_j$ and $V_j^\perp$ are $K_j$--invariant. Thus
\[|\langle e_j^\perp , \pi_j (x)e_j^\perp\rangle|\leqq \|e_j^\perp\|^2\to 0\, .\]
Hence $\varphi_{\mu_n}(x)\to \varphi_{\mu_\infty}(x)$, i.e., $\varphi_{\mu_n}
\to \varphi_{\mu_\infty}$ pointwise.
Similarly, for $x,y\in G_j$,
\begin{eqnarray*}
\lim_{j\to \infty} \int_{K_j} \varphi_{\mu_\infty} (xky) \, dk
& = &\lim_{j\to \infty}\left(\int_{K_j} \varphi_{\mu_j} (xky)\, dk+\int_{K_j}
 \langle e_j^\perp , \pi_j(xky)e_j^\perp \rangle\, dk\right)\\
&=& \lim_{j\to \infty}\varphi_{\mu_j}(x)\varphi_{\mu_j}(y)+\lim_{j\to \infty}\int_{K_j} \langle e_j^\perp , \pi_j(xky)e_j^\perp\rangle \, dk\\
&=&\varphi_{\mu_\infty}(x)\varphi_{\mu_\infty}(y)
\end{eqnarray*}
because
\[
\left | \int_{K_j} \langle e_j^\perp , \pi_j(xky)e_j^\perp \, dk|\right |
\leqq \|e_j^\perp\|^2\int_{K_j}\, dk=\|e_j^\perp\|^2\to 0\, .\qedhere\]
\end{proof}

The definition and construction of spherical functions in the infinite
rank case remains to be clarified.

\section{The Finite Rank Case}\label{sec7}
\setcounter{equation}{0}
\setcounter{theorem}{0}

\noindent
In this section we discuss the case of the finite rank Grassmannian in
more detail. In this case we can assume that
$\fa_n=\fa$ is fixed for all $n$.  Then $\Sigma_n=\Sigma$ is fixed, and
only the root multiplicities change as $n$ grows.

We start with the case of the sphere.
Let $X_n=S^n=SO(n+1)/SO(n)$ where the inclusions are given by
$S^n\hookrightarrow S^{n+1}\, ,\quad u\mapsto (u,0)$
and
\[SO(n)\hookrightarrow SO(n+1)\, ,\quad A\mapsto
	\begin{pmatrix} A & 0 \\ 0 & 1\end{pmatrix}\, .\]

Let $\cH_n (k)$ be the space of homogeneous harmonic polynomials on
$\R^{n+1}$ of degree $k$ with inner product
\[ \langle p, q\rangle =\tfrac{1}{k!} \partial_{\overline{q}}(p)(0)\, ,\]
and let $\pi_{k,n}$ denote the natural representation of $G_n=SO (n+1)$
on $\cH_n(k)$.  Then $\pi_{k,n}$ is a spherical representation of $G_n$ and
every spherical representation is constructed this way.  Clearly
$\cH_n (k)\subset \cH_m(k)$ if $n\leqq m$ and the natural inclusion is an
isometry. We can take $u (x)=(x_1-ix_2)^k$ as a highest weight vector.
It is independent of $n$. Use polar coordinates
$\cos (\theta )e_1+\sin (\theta )u$, $u\in S^{n-1}$. Then the
$K_n$-invariant functions corresponds to functions of one variable
$P(\cos\, \theta)$.
Using the radial component of the Laplacian on $S^n$, the spherical function
associated to $\cH_n(k)$ is a solution to the initial value problem
\begin{eqnarray*}
\left(\dfrac{d^2}{d\theta^2}+(n-1)\cot (\theta) \dfrac{d}{d\theta}\right)P_{n,k}(\cos \theta)
&=& - k(k+n-1)P_{n,k}(\cos \theta)\\
P_{n,k}(1)&=&1\,  .
\end{eqnarray*}
Putting $t=\cos (\theta)$ and dividing by $n$,
\[
\left(\frac{(1-t^2)}{n}\dfrac{d^2}{dt^2}- t\dfrac{d}{dt}\right)p_{n,k}(t)
=-\left(\frac{k(k-1)}{n}+k\right)p_{n,k}(t)\, ,\quad p_{n,k}(1)=1\]
Letting $n\to \infty$ we see that the corresponding spherical function
$\varphi_\infty$ is a solution to the first order differential equation
\[ t \dfrac{d}{dt}p_{\infty ,k}(t)
=k p_{\infty,k}(t)\, ,\quad p_{\infty,k}(1)=1\]
Thus $\varphi_{\infty, k}(x) =x_1^k$.
According to \cite{F2008}, p. 11, every $K_\infty$-spherical function on $G_\infty $ is given by
\[  g\mapsto \langle {e_1},{g(e_1)}\rangle ^k \text{ for some integral }
k\geqq 0\, .\]
Thus  all  spherical functions or $G_\infty$ are constructed as in our
limit theorem. In particular, all irreducible spherical representations are
obtained by the limit construction. The question is whether that is also
the case for the other finite rank Grassmannians.  We have been informed
that this question has been answered affirmative in
\cite{RKV}: Every spherical function in the sense of
Olshanskii is a limit of spherical functions on $X_n$ obtained by letting
the root multiplicities go to infinity. Combining our construction with this
result gives the following theorem, which in principle states that the
theory of highest weights remains valid for the finite rank case.

\begin{theorem} Assume that the rank of $X_\infty$ is finite. Let $\pi$
be an unitary irreducible $K_\infty$-spherical representation of $G_\infty$.
Then there exists $\mu\in \Lambda^+$ such that $\pi\simeq \pi_{\mu,\infty}$.
In particular the set of equivalence classes of unitary $K_\infty$-spherical
representation of $G_\infty$ is isomorphic to $\Lambda^+$.
\end{theorem}

This gives us in particular a natural embedding of $V_{\mu_\infty}$ into
$C_b (X_\infty)$, the space of bounded continuous functions on $X_\infty$ by
$u\mapsto \langle v, \pi_{\mu_\infty}(x)e_{\mu_\infty}\rangle$.
We can then ask the question: If $(\pi,V)$ is a unitary irreducible
representation of $G_\infty$ with $V\subset C_b(X_\infty)$ does there exists
a $\mu\in \Lambda^+$ such that
$(\pi,V)\simeq (\pi_{\mu,\infty},V_{\mu,\infty})$?

It is also natural to ask what happens in the infinite rank case. For that
we would like to point the following out. Fix $\mu_n\in \Lambda^+$ and
$e_n^*\in (V_n^*)^{K_n}$, $\|e_n\|=1$. Then
construct the inductive sequence $(\pi_m,V_m)$ as before. Let
$e_m^*\in (V_m^*)^{K_m}$ be so that the projection of $e_m^*$ onto $V_n^*$ is our
fixed $K_n$-invariant functional $e_n^*$. Then the sequence $\{e_m^*\}$ defines
an element $e_\infty^*\in \varprojlim V_m^*\simeq (\varinjlim V_m)^*$, where the limit
 is now with respect to the inductive/projective lmit topology in not in the category of
Hilbert spaces. In particular, $e_\infty^*$ defines a linear form on
$\bigcup V_n$ given by $u\mapsto \langle {u},{e_n^*}\rangle$ if $u\in V_n$. It therefore defines a linear $G_\infty$-equivariant embedding of
$\bigcup V_n$ into $C_b(X_\infty)$.

Restricting $\pi_{\mu_\infty}$ to $G_n$ gives a unitary representation of
$G_n$. Let $V_{\mu_\infty}^{n\infty}$ denote the space of smooth vectors
for this representation and let
\[V_{\mu_\infty}^\infty =\bigcap_n V_{\mu_\infty}^{n\infty}\, .\]
Then $V_{\mu_\infty}^\infty$ is a locally convex topological vector space in
the usual way. A continuous linear form $\nu: V_{\mu_\infty}^\infty \to \C$ is
a \textit{distribution vector}. We denote the space of distribution vectors
by $V_{\mu_\infty}^{-\infty}$. The question now is whether
$e_\infty^*: v \mapsto \langle v, e_\infty^*\rangle$ is a distribution vector.

\end{document}